\documentclass[11pt]{amsart}

\usepackage{amsmath,amssymb,amsthm,amsfonts,mathtools}
\usepackage[shortlabels]{enumitem}
\setlist{font=\normalfont}
\usepackage{microtype}
\usepackage{cite}
\usepackage[hidelinks]{hyperref}
\usepackage{url}
\newtheorem{theorem}{Theorem}[section]
\newtheorem{proposition}[theorem]{Proposition}
\newtheorem{lemma}[theorem]{Lemma}
\newtheorem{corollary}[theorem]{Corollary}

\theoremstyle{definition}
\newtheorem{definition}[theorem]{Definition}

\theoremstyle{remark}
\newtheorem{remark}[theorem]{Remark}

\newcommand{\F}{\mathbb F}

\newcommand{\HD}[1]{\mathcal D^{(#1)}}

\title{A complete classification of minimal value set polynomials}

\author{Herivelto Borges}
\address{Instituto de Matem\'atica e Estat\'istica, Universidade de S\~ao Paulo\\
Rua do Mat\~ao, 1010, S\~ao Paulo, SP 05508-090, Brazil}
\email{hborges@ime.usp.br}

\thanks{The first author was supported by CNPq (Brazil), grants 406377/2021-9 and 310194/2022-9, and by FAPESP (Brazil), grant 2022/15301-52.}

\author{Lucas Reis}
\address{Departamento de Matem\'atica, Universidade Federal de Minas Gerais, UFMG, Belo Horizonte, MG 30270-901, Brazil}
\email{lucasreismat@gmail.com}
\thanks{The second author was supported by CNPq (Brazil), grants 310583/2025-0 and 420721/2025-8, by FAPESP (Brazil), grant 25/21886-4 and by FAPEMIG (Brazil), grant APQ-02427-26.}

\begin{document}

\begin{abstract}
We give a complete classification of minimal value set polynomials (MVSPs),
settling a problem introduced by Carlitz, Lewis, Mills and Straus in 1961.
For MVSPs $F \in \mathbb{F}_q[x]$ with more than two values, earlier work showed that their value
sets $\left\{F(a): a \in \mathbb{F}_q\right\}$ can be reduced to the form \(\mathcal U^v=\{u^v:u\in \mathcal U\}\), where \(\mathcal U\) is a  vector space
over a subfield of \(\mathbb F_q\), and \(v\geq 1\) is a suitable integer. This classification proceeds through two parallel constructions,
according to whether $\mathcal U$ is itself a field or not, governed respectively by the cyclic multiplicative structure and by the additive linearized structure. In both cases, Hasse derivatives impose constraints on the root multiplicities
of a distinguished factor of $F$, leading to a factorization that preserves
the underlying structure in a polynomial of smaller degree and thus allows
the argument to be iterated.
\end{abstract}

\maketitle

\section{Introduction}

Let \(q\) be a prime power, and let \(F\in\F_q[x]\) be a nonconstant
polynomial of degree \(d<q\). Its \emph{value set}
\(V_F:=\{F(a):a\in\F_q\}\) satisfies
\begin{equation}\label{eq:range}
\left\lceil\frac{q}{\deg F}\right\rceil
\leq \#V_F\leq q.
\end{equation}
A fundamental problem is to determine which cardinalities occur in this range and how \(V_F\) constrains the structure of \(F\),
with the ultimate goal of classifying  the polynomials with a prescribed value-set size.

The problem of determining the number of values assumed by a polynomial
over a finite field goes back at least to the work of Carlitz and  Birch
and Swinnerton-Dyer's study of a problem of Chowla
\cite{Carlitz1955,BSD}. Subsequent work showed that these cardinalities
are highly constrained. Cohen proved, as a consequence of Chebotarev
density, that for every fixed positive integer \(d\) there is a finite
set \(T_d\subset\mathbb Q_{>0}\) such that for every prime power $q$ and every
\(F\in\F_q[x]\) of degree \(d\)
\[
\#V_F=c_Fq+O_d(\sqrt q)
\]
for some \(c_F\in T_d\) \cite{Cohen,MWW}. Thus, even asymptotically,
the proportion \(\#V_F/q\) can take values only near a finite set of
possibilities.

Different parts of this spectrum have been studied by several methods.
Wan, Shiue and Chen obtained general upper and lower bounds for
\(\#V_F\) in terms of polynomial invariants \cite{WSC}, while Guralnick
and Wan used permutation-group methods to obtain stronger restrictions
for non-permutation polynomials, including asymptotic bounds when the
degree is prime to the characteristic \cite{GW}. More refined problems,
in which one prescribes not only the value set but also the
multiplicities of its elements, have also been considered \cite{MRW}.
Near the lower end of the range, Gomez-Calderon and Madden gave a
complete list of polynomials in \(\F_q[x]\) of degree \(d<q^{1/4}\)
whose value set has size less than twice the minimum possible
\cite{Cald2}.

The two extremal values of the range \eqref{eq:range} give rise to two
distinguished classes. At the upper endpoint, \(\#V_F=q\), one obtains
permutation polynomials. The related theory of exceptional polynomials
has led to classification results involving arithmetic and geometric
monodromy and finite permutation groups. In particular, Guralnick and
Zieve classified polynomials with prescribed
\(\operatorname{PSL}_2\)-monodromy under a ramification hypothesis and,
together with the companion work of Guralnick, Rosenberg and Zieve,
completed the classification of indecomposable exceptional polynomials
whose degree is not a power of the characteristic \cite{GZ,GRZ}.

The polynomials attaining equality at the lower endpoint of \eqref{eq:range}, that is, the nonconstant polynomials \(F\in\F_q[x]\) for which
\[
\#V_F=\left\lceil\frac{q}{\deg F}\right\rceil,
\]
are called \emph{minimal value set polynomials}, or MVSPs. The primary goal of this paper is to complete their classification over arbitrary finite fields by settling the case of more than two values. 

These polynomials can be viewed as polynomial maps with exceptionally large rational fibers, in analogy with the highly split fibers used over number fields to construct curves with many rational points \cite{Elkies}. They are also connected with Frobenius nonclassical curves, Castle curves and certain recursive towers of function fields, with applications to finite geometry and algebraic-geometric coding theory \cite{Bor0,Bor1,Toledano,GPTU,MunueraSepulvedaTorresCastle}. One consequence of our classification is a complete solution to the problem of determining the Frobenius nonclassical components of curves of the form $F(x)=G(y)$ (see Section~\ref{sec:fnc}).

The systematic study of MVSPs was initiated by Carlitz, Lewis, Mills and Straus in 1961, in connection with a generalization of Waring's problem modulo a prime \cite{CLMS}. Mills subsequently developed a  structural theory and suggested that the families then known might already exhaust all possible MVSPs \cite{Mills}.

For a nonempty subset $S\subseteq\F_q$, write
\[
\mathcal P(S,q)
=
\{F\in\F_q[x]:F\text{ is an MVSP and }V_F=S\}.
\]
These two pioneering works determined the sets $\mathcal P(S,q)$ for $\#S\geq\sqrt q$ and established a complete classification over $\F_p$ and $\F_{p^2}$, where $p$ is prime. The case $\#S\leq 2$, already discussed in \cite{CLMS}, is essentially trivial and will not be treated here.

Despite these advances, the general classification problem remained open and was explicitly formulated in \cite[Problem~8.3.10]{Mullen}.

Substantial progress was obtained for several families. In 2013, Borges and Concei\c{c}\~ao described the sets $\mathcal P(\F_{q'},q)$ for subfields $\F_{q'}\subseteq\F_q$. They also constructed new families with other prescribed value sets, showing that the families suggested by Mills did not exhaust all MVSPs \cite{Bor2}. Further cases were subsequently settled, including the complete classification over fields of size $p^3$ \cite{BR21}.

A recent different approach, due to Deng, Xu and Zhu, uses arithmetic and geometric monodromy groups \cite{DXZ}. They study the Galois subclass of \emph{absolutely minimal value set polynomials} and replicate, via monodromy, classification results of Mills and of Borges and Concei\c{c}\~ao. This provides a group-theoretic explanation of those results, but does not determine $\mathcal P(S,q)$ in general. In the remaining range $\#S<\sqrt q$, a complete classification must also account for polynomials outside the Galois subclass.

More recently, a structural reformulation of the problem was
given in \cite{BR}. There, it was shown that every subset
$S\subseteq\F_q$ with $\#S>2$ and $\mathcal P(S,q)\ne\emptyset$
has the form
\[
a\cdot \mathcal U^v+b:=\{au^v+b:u\in  \mathcal U\},
\qquad a\in\F_q^\times,\quad b\in\F_q,
\]
where $\mathcal U\subseteq\F_q$ is a vector space over a subfield $\F_{q'}\subseteq \F_q$ with $1\in \mathcal U$ and $v$ divides $q'-1$. After an affine change
of values, the problem is, therefore, to describe
$\mathcal P(\mathcal U^v,q)$ for $\mathcal U,v$ as above. The space $\mathcal U$ provides
the additive structure, while the passage to $\mathcal U^v$ introduces
the multiplicative one. Determining the value sets, however, does not itself characterize the polynomials realizing them.

That work \cite{BR} also provided a complete description of
$\mathcal P(\mathcal U,q)$ and conjectured that the remaining
cases could be reduced to MVSPs with simpler value sets.
More precisely, every polynomial in
$\mathcal P(\mathcal U^v,q)$ was conjectured to have the form
$H^t$, where $H$ is another MVSP with a simpler value set
and $t$ is a suitable integer. The main difficulty is
therefore to translate the additive and multiplicative
structure of the value set into a corresponding
factorization of each $F\in\mathcal P(\mathcal U^v,q)$.

In this paper, we prove this conjecture
(Theorem~\ref{thm:conjecture}). As a consequence, we derive
the following explicit characterization of all MVSPs.

\begin{theorem}\label{thm:explicit}
Let $q=p^n$ be a prime power and $F\in\F_q[x]$
with $\#V_F>2$. Then $F$ is an MVSP if and only if,
up to affine changes of variables and values, it has the form
\[
F(x)=A(P(x))^v,
\]
where
\begin{enumerate}[{\upshape(i)}]
    \item \(P\in\F_q[x]\) is an MVSP whose value set is a
    subfield \(\F_{q'}\subseteq\F_q\);

    \item for some subfield \(\F_{p^k}\subseteq\F_{q'}\),
    \(A(x)=\prod_{w\in W}(x-w)\), where
    \(W\subseteq\F_{q'}\) is an \(\F_{p^k}\)-vector subspace;

    \item \(v\mid(p^k-1)\) and \(2v\deg A<q'\).
\end{enumerate}
In this case,
$\#V_F
=
1+\frac{q'/\deg A-1}{v}.$
\end{theorem}

The MVSPs with subfield value sets have already been
explicitly classified (Theorem~\ref{thm:b-c}),
while the remaining ingredients in the theorem are
vector subspaces and power maps. Thus
Theorem~\ref{thm:explicit} provides a complete
description of all MVSPs over any given finite field.

Our proof does not rely on the existing classification of
MVSPs in the range $\# S\ge \sqrt q$. Its starting point is
Mills's partial decomposition, together with the
differential equation
\[
T(F)=\theta(x^q-x)F',
\]
where $\theta\in\F_q^\times$, and
$T(x)=\prod_{a\in V_F}(x-a)$ is the vanishing
polynomial of the value set of an MVSP $F\in\F_q[x]$.

The method for solving this differential equation is developed
in Section~\ref{sec:mills-equation}, where we introduce the notion of
\emph{covering systems} and establish a more general result
that includes the present setting as a special case.

\section{Explicit descriptions of  fundamental MVSPs}
We begin with MVSPs whose value set is a subfield of $\F_q$. The results of Borges and Concei\c{c}\~ao in~\cite{Bor2} yield the following description of these polynomials in terms of traces of monomials.

\begin{theorem}\label{thm:b-c}
Let $q=p^n$ be a prime power, fix $e|n$, and set $s=p^e$. If $s>2$, then $\mathcal P(\F_{s}, q)$ comprises the nonconstant polynomials $F\in \F_{q}[x]$ of degree $\le q-1$ such that 
$$F(x)\equiv \sum_{g\in \mathcal D}\mathrm{Tr}_{n/e}(a_gx^{g})\pmod {x^q-x},$$
where $a_g\in \F_q$, $\mathrm{Tr}_{n/e}(x)=\sum_{i=0}^{n/e-1}x^{s^i}$ denotes the trace map from $\F_{q}$ to $\F_s$ and 
$\mathcal D=\{\varepsilon_0 s^0+\cdots +\varepsilon_{n/e-1}s^{n/e-1}:\varepsilon_i=0, 1\}$ is the set of nonnegative integers less than $q$ whose $s$-adic expansions contain only the digits $0$ and $1$.
\end{theorem}

\begin{proof}
Let $F\in \F_q[x]$ be a polynomial satisfying the above. The congruence above easily yields the inequalities
$\frac{q}{s}\le \deg(F)\le \frac{q-1}{s-1}$. In particular, $\# V_F\ge \left\lceil\frac{q}{\deg F}\right\rceil= s$. Moreover, since the image set of $\mathrm{Tr}_{n/e}(a_gx^{g})$ over $\F_q$ lies in $\F_{s}$, we obtain $V_F\subseteq \F_s$. Thus $F\in \mathcal P(\F_{p^e}, q)$.

Conversely, let $F\in\mathcal P(\F_s,q)$. By Lemma 4.6 in~\cite{Bor2}, the degree of every monomial occurring in $F$ belongs to $\mathcal D$. Thus we may write
$F(x)=\sum_{g\in\mathcal D}c_gx^g$, where $c_g\in\F_q$. Choose $\beta\in\F_q$ such that $\mathrm{Tr}_{n/e}(\beta)=1$.
Since $F(a)\in\F_s$ for every $a\in\F_q$, we have
$\mathrm{Tr}_{n/e}(\beta F(a))=F(a)$. Consequently,
$$
F(x)\equiv \mathrm{Tr}_{n/e}(\beta F(x))
=\sum_{g\in\mathcal D}\mathrm{Tr}_{n/e}(\beta c_gx^g)
\pmod{x^q-x}.
$$
Taking $a_g=\beta c_g$ gives the desired representation.
\end{proof}

The subfield case also determines the MVSPs whose value set is a vector subspace. More precisely, the following result
from~\cite{BR} expresses these polynomials as compositions of
a fixed additive polynomial with MVSPs whose value set is the
smallest subfield containing the prescribed subspace.

\begin{theorem}[{\cite[Theorem~2.4]{BR}}]\label{thm:subspace}
Let $q=p^n$ be a prime power, let $k|n$, and let $\mathcal U\subseteq \F_q$ be an $\F_{p^k}$-vector space with $\#\mathcal U> 2$ and $1\in \mathcal U$. Write $L_{\mathcal U}(x)=\prod_{u\in \mathcal U}(x-u)$. If $\F_{p^{km}}\subseteq \F_q$ is the smallest subfield of $\F_q$ containing $\mathcal U$, then there exists a unique $p^k$-additive polynomial $M_{\mathcal U}(x)$ such that $M_{\mathcal U}(L_{\mathcal U}(x))=x^{p^{km}}-x$. In this case, 
$$\mathcal P(\mathcal U, q)=\{M_{\mathcal U}(f(x)): f\in \mathcal P(\F_{p^{mk}}, q)\}.$$
\end{theorem}

\begin{remark}\label{rem:ad-div}
If $L_{\mathcal U}, M_{\mathcal U}$ are as in Theorem~\ref{thm:subspace}, Lemma 3.3 in~\cite{BR} implies that $M_{\mathcal U}(L_{\mathcal U}(x))=L_{\mathcal U}(M_{\mathcal U}(x))$. In particular, $M_{\mathcal U}(x)$ divides $x^{p^{km}}-x$. 
\end{remark}

We conclude this section by stating the conjecture proposed
in~\cite{BR} as a theorem, whose proof is given in Section~\ref{sec:proof}.

\begin{theorem}\label{thm:conjecture}
Let $q=p^{n}$, $k$ a divisor of $n$, and $v$ a divisor of $p^k-1$. Let $\mathcal U\subseteq \F_q$ be an $\F_{p^k}$-vector space of dimension $m$ such that $1\in\mathcal U$ and $\#\mathcal U^v=r+1>2$. Then the following hold:

\begin{enumerate}[\rm(i)]
    \item if $\mathcal U$ is not a field, then 
    $\mathcal P(\mathcal U^v, q)=\{f^v: f\in \mathcal P(\mathcal U, q)\}.$
    \item if $\mathcal U$ is a field, then 
    $$\mathcal P(\mathcal U^v, q)=\{f^{(p^e-1)/r}: f\in \mathcal P(\mathbb F_{p^e}, q)\},$$
    where $\F_{p^e}$ is the smallest subfield of $\F_q$ that contains $\mathcal U^v$.
\end{enumerate}
\end{theorem}

\section{Preliminaries}
In this section we provide some preliminary results regarding MVSPs, Hasse derivatives and binomial coefficients modulo a prime.

\subsection{On MVSPs}
We recall some structural results on MVSPs that will be used
throughout the paper. We begin with Mills's theorem, which provides a special factorization for MVSPs. 

\begin{theorem}[Mills]\label{thm:mills}
Let $F\in \F_q[x]$ be an MVSP with value set $V_{F}=\{\gamma_0, \ldots, \gamma_r\}$, where $\# V_F=r+1>2$ and $\# (F^{-1}(\gamma_0)\cap \F_q)\le \# (F^{-1}(\gamma_i)\cap \F_q)$ for  $i=1,\ldots, r$. Set $L=\gcd(F-\gamma_0,x^q-x)$. Then there are positive integers $\mu, \kappa, \nu$  and polynomials $A, B, N\in \F_q[x]$, with $L\nmid N$, such that the following hold.
\begin{enumerate}[\rm(i)]
\item $\F_{p^{\kappa}} \subseteq \F_q$, $\nu$ divides $p^{\kappa}-1$, and $\nu r+1=p^{\mu\kappa}$.
\item $F=L^{\nu}N^{p^{\mu\kappa}}+\gamma_0$.
\item $L=A^{p^{\mu\kappa}}x+B^p$.
\end{enumerate}
\end{theorem}

The following result  is proved in  \cite[Theorem 3.1]{Bor2} and  \cite[Lemma 2.4]{Bor1}.

\begin{theorem}\label{MillsBor}
Let $F \in \F_q[x]$ be a polynomial with value set $V_{F}=\{\gamma_0, \ldots, \gamma_r\}$, where $r>1$.
Then $F$ is an MVSP if and only if there is a monic polynomial $T \in \F_q [x]$ and $\theta \in \F_q^\times$ such that
\begin{equation}\label{MVSP}
T(F)=\theta(x^q-x)F^{\prime}.
\end{equation}
Furthermore, the polynomial $T$ is uniquely determined as $T=\prod\limits_{i=0}^{r}(x-\gamma_i)$, and if $\gamma_0, \ldots, \gamma_r$ are ordered as in Theorem
\ref{thm:mills}, then $\theta=-T^{\prime}\left(\gamma_{i}\right)$ for all $\gamma_{i} \in V_{F}\backslash \{\gamma_0\}$.
\end{theorem}

We next recall which subsets of $\F_q$ with more than two
elements occur as value sets of MVSPs. This characterization
explains the role of powers of vector subspaces in
Theorem~\ref{thm:conjecture}.

\begin{lemma}[{\cite[Theorem~2.2]{BR}}]\label{lem:br-1}
For each $S\subseteq \F_q$ with $\# S>2$, we have $\mathcal P(S, q)\ne \emptyset$ if and only if 
$S=a\cdot \mathcal U^{v}+b$ for some $a\in\F_q^\times$, $b\in\F_q$, an $\F_{p^k}$-vector subspace $\mathcal U\subseteq\F_q$ containing $1$, and a positive divisor $v$ of $p^k-1$, where $\F_{p^k}\subseteq\F_q$. 
\end{lemma}

For value sets of the form $\mathcal U^v$, the following
result identifies the distinguished value $\gamma_0$
in Mills's theorem under suitable assumptions.

\begin{lemma}[{\cite[Lemma~6.1]{BR}}]\label{lem:gamma}
Let $\mathcal U, v$, and $k$ be as in Theorem~\ref{thm:conjecture}. Suppose that $F\in \mathcal P(\mathcal U^v, q)$, and let $\gamma_0 \in V_F=\mathcal U^v$ be as in Theorem~\ref{thm:mills}. If $v>1$ and $\mathcal U^v$ is not a subfield of $\F_q$, then $\gamma_0=0$.
\end{lemma}

In the next two lemmas, we provide connections between certain MVSPs and their powers.

\begin{lemma}[{\cite[Lemma~2.4]{BR}}]\label{lem:br-v}
    Let $\mathcal U$ and $v$ be as in Theorem~\ref{thm:conjecture}, and let $F\in\F_q[x]$. Then 
    $F\in \mathcal P(\mathcal U, q)$ if and only if $F^v\in\mathcal P(\mathcal U^v, q)$.
\end{lemma}

\begin{lemma}\label{lem:reduction}
    Let $\mathcal U, v$ and $k$ be as in Theorem~\ref{thm:conjecture}. Let $w$ be a positive integer. 
    If $vw$ divides $p^k-1$, $\# \mathcal U^{vw}>2$, and 
    $$\mathcal P(\mathcal U^{vw}, q)=\{f^{vw}: f\in \mathcal P( \mathcal U, q)\},$$
    then $\mathcal P(\mathcal U^{v}, q)=\{f^{v}: f\in \mathcal P(\mathcal U, q)\}$.
\end{lemma}

\begin{proof}
Set $\ell=\#\mathcal U^{w}-1$ and $h=\#\mathcal U^{vw}-1=\frac{\ell}{v}$. If $f\in \mathcal P(\mathcal U^{w}, q)$, then 
$$\ell+1=\left\lceil\frac{q}{\deg f}\right\rceil\iff \ell< \frac{q}{\deg f}\le \ell+1.$$ In particular, 
$h< \frac{q}{\deg f^{v}}\le h+1$ and we have $V_{f^{v}}=\mathcal U^{vw}$, that is, $f^v\in \mathcal P(\mathcal U^{vw}, q)$.
Therefore,  
$$\{f^v: f\in \mathcal P(\mathcal U^{w}, q)\}\subseteq \mathcal P(\mathcal U^{vw}, q).$$
Applying the preceding argument with $w=1$ gives the inclusion
$\{f^v: f\in \mathcal P(\mathcal U, q)\}\subseteq \mathcal P(\mathcal U^{v}, q)$. Conversely, fix $F\in \mathcal P(\mathcal U^v, q)$. From the first part of the proof, we obtain $F^{w}\in \mathcal P(\mathcal U^{vw}, q)$. By hypothesis, there exists $G\in \mathcal P(\mathcal U,q)$ such that $F^w=G^{vw}$. Hence $F=c_0G^v$, where $c_0^w=1$. Since $vw$ divides $p^k-1$ and $\F_{p^{k}}^\times$ is cyclic, it follows that $c_0=c^v$ for some $c\in \F_{p^k}^\times$. In particular, $F=(c\cdot G)^v$. Now, since $\mathcal U$ is an $\F_{p^k}$-vector space, we have $c\cdot \mathcal U=\mathcal U$. In particular, $G\in \mathcal P(\mathcal U,q)$ implies $c\cdot G\in \mathcal P(\mathcal U,q)$. Thus $F=(c\cdot G)^v$ with $c\cdot G\in\mathcal P(\mathcal U,q)$.
\end{proof}

\begin{remark}\label{rem:field}
Let $\mathcal U$ and $v$ be as in Theorem~\ref{thm:conjecture}. According to \cite[Remark~2.5]{BR}, Theorem~\ref{thm:conjecture} holds if $\mathcal U^v$ is a field. 
\end{remark}

\subsection{Hasse derivatives}

We recall basic properties of Hasse derivatives and conclude this
subsection with a lemma on their effect on zeros and poles of rational
functions. For the standard properties and further details on Hasse
derivatives, we refer to \cite[Chapter~5]{HKT}.

Let $K$ be a field, and let $K((x))$ denote the field of formal Laurent series
in the indeterminate $x$.
For $u\in\mathbb Z$ and $k\ge1$, we use the generalized binomial
coefficient
\[
\binom{u}{k}:=\frac{u(u-1)\cdots(u-k+1)}{k!}\in\mathbb Z,
\]
with $\binom{u}{0}=1$.

The following lemma records consequences of Lucas' theorem
\cite[Lemma~A.6]{HKT}; we omit its proof. For negative upper indices,
we use the fact
that, for each integer $a\ge1$ and $0\le k<p^a$, the residue of
$\binom{u}{k}$ modulo $p$ depends only on $u\bmod p^a$.

\begin{lemma}\label{lem:binomial}
Let $p$ be a prime and $s=p^e$, where $e\ge1$.
\begin{enumerate}[(i)]
    \item For $u\in\mathbb Z$ and $k\ge0$ with $s\nmid k$,
    we have $\binom{su}{k}\equiv0\pmod p$.

    \item Let $u\in\mathbb Z$. Then
    $\binom{u}{k}\equiv0\pmod p$ for $2\le k<s$
    if and only if $u\equiv0,1\pmod s$.

    \item If $M\ge2$ is an integer such that
    $\binom{-M}{k}\equiv0\pmod p$ for $2\le k\le M$,
    then $M=p^h-1$ for some integer $h\ge1$.
\end{enumerate}
\end{lemma}

\begin{definition}\label{def:hasse}
For each $k\ge0$, the $k$-th Hasse derivative is the map
$\HD{k}:K((x))\to K((x))$ defined by
\[
\HD{k}\!\left(\sum_{u\ge u_0}a_u x^u\right)
=
\sum_{u\ge u_0}\binom{u}{k}a_u x^{u-k}.
\]
\end{definition}

Hasse derivatives satisfy the following product rule
\cite[Lemma~5.72, Eq.~(5.23)]{HKT}.

\begin{lemma}\label{lem:hasse-product}
For $f,g\in K((x))$ and $k\ge0$, we have
\begin{equation}\label{eq:hasse-product}
\HD{k}(fg)=\sum_{i=0}^k\HD{i}f\cdot\HD{k-i}g.
\end{equation}
\end{lemma}


By the product rule, Hasse derivatives preserve $K(x)\subseteq K((x))$.
Their restrictions to $K(x)$ satisfy the following properties, which
follow directly from Definition~\ref{def:hasse},
Lemma~\ref{lem:binomial}(i), and Lemma~\ref{lem:hasse-product}.

\begin{remark}\label{rem:01}
Suppose $\operatorname{char}K=p>0$, and let $s=p^e$,
where $e\ge1$.
\begin{enumerate}[(i)]
    \item For $F\in K(x)$, we have
    $\HD{k}(F^s)=0$ for $1\le k<s$.

    \item For $A,B\in K(x)$, we have
    $\HD{k}(A^sx+B^s)=0$ for $2\le k<s$.
\end{enumerate}
\end{remark}







\begin{remark}\label{rem:vanish}
For a nonzero rational function $f\in K(x)$, let $v_\alpha(f)$
denote its order at $\alpha\in K\cup\{\infty\}$.
We set $v_\alpha(0)=+\infty$.

Using the product rule and the binomial expansion for polynomials,
one can check that Hasse derivatives commute with translation on
rational functions:
$\HD{k}(f(x+\alpha))=(\HD{k}f)(x+\alpha)$ for $\alpha\in K$.
Thus, for nonzero $f\in K(x)$, applying Definition~\ref{def:hasse}
to the Laurent expansion of $f(x+\alpha)$ at zero and replacing
$x$ by $x-\alpha$, we obtain
\[
f=\sum_{j\ge m}c_j(x-\alpha)^j
\quad\Longrightarrow\quad
\HD{k}f=\sum_{j\ge m}\binom jk c_j(x-\alpha)^{j-k}
\quad\text{in }K((x-\alpha)),
\]
where $m=v_\alpha(f)$ and $c_m\ne0$. Consequently,
\[
v_\alpha(\HD{k}f)\ge v_\alpha(f)-k.
\]
Moreover, $\HD{k} f$ is regular at $\alpha$ whenever $f$ is,
since $\binom{j}{k}=0$ for $0\le j<k$.

At infinity, write $f=g/h$ with $g,h\in K[x]$ nonzero.
By Definition~\ref{def:hasse},
$v_\infty(\HD{i}a)\ge i-\deg a$ for every nonzero polynomial $a$.
We proceed by induction on $k$, the case $k=0$ being immediate.
For $k\ge1$, Eq.~\eqref{eq:hasse-product} gives
\[
h\HD{k}f
=\HD{k}g-\sum_{i=1}^k\HD{i}h\,\HD{k-i}f.
\]
By induction, each term in the sum has valuation at least
$(i-\deg h)+(v_\infty(f)+k-i)=k-\deg g$.
The same bound holds for $\HD{k}g$. Dividing by $h$, we obtain
\[
v_\infty(\HD{k}f)\ge k-\deg g+\deg h=v_\infty(f)+k.
\]

Suppose $\operatorname{char}K=p>0$.
If $f$ has a pole of order $M\ge1$ at a finite point
and $\HD{k} f=0$, then $\binom{-M}{k}\equiv0\pmod p$.
If instead $f$ has a pole of order $M$ at infinity
and $\HD{k} f=0$, then $\binom{M}{k}\equiv0\pmod p$.
\end{remark}

We use the convention that the zero polynomial has degree $-\infty$.
For a nonzero polynomial $N\in K[x]$, we write $\operatorname{rad}(N)$
for the product of its distinct monic irreducible factors.
The preceding estimates yield the following bound.






\begin{lemma}\label{lem:hasse}
Let $N\in K[x]\setminus\{0\}$, and set $n=\deg N$,
$R=\operatorname{rad}(N)$, and $\rho=\deg R$.
Let $d,r$ be nonnegative integers, and let $Z\in K[x]$
satisfy $\deg Z\le d$.
Then, for all integers $k\ge0$ and $r_0$ with $0\le r_0\le r$,
\begin{equation}\label{eq:cleared-Hasse}
H_k(Z):=N^rR^k\HD{k}\!\left(\frac{Z}{N^{r_0}}\right)
\in K[x],
\end{equation}
and $\deg H_k(Z)\le d+(r-r_0)n+(\rho-1)k$.
\end{lemma}

\begin{proof}
If $H_k(Z)=0$, the result is immediate.
Assume that $H_k(Z)\ne0$, and set $f=Z/N^{r_0}$. For each root $\lambda\in\overline{K}$ of $N$, we have
$v_\lambda(R)\ge1$, and Remark~\ref{rem:vanish} gives
\[
\begin{aligned}
v_\lambda(H_k(Z))
&=r\,v_\lambda(N)+k\,v_\lambda(R)+v_\lambda(\HD{k} f)\\
&\ge (r-r_0)v_\lambda(N)+k\bigl(v_\lambda(R)-1\bigr)\\
&\ge0.
\end{aligned}
\]
Away from the roots of $N$, $f$ and its Hasse derivatives
have no finite poles.
Thus $H_k(Z)$ has no finite poles and belongs to $K[x]$. 

The bound on $\deg H_k(Z)$ follows directly from
$\deg H_k(Z)=-v_\infty(H_k(Z))$,
$v_\infty(\HD{k} f)\ge v_\infty(f)+k$, and
$v_\infty(f)=r_0n-\deg Z$.
\end{proof}




\section{A general approach to Mills' equation}\label{sec:mills-equation}

In earlier works on the complete characterization of MVSPs $F\in\F_{p^i}[x]$ for $i=3,4$, the proofs essentially reduced to the case $V_F=\mathcal C_r\cup\{0\}$, where $\mathcal C_r$ is the multiplicative subgroup of order $r$ and $r\mid p-1$. In this case, Theorems~\ref{thm:mills} and~\ref{MillsBor} give $F=(A^px+B^p)^{(p-1)/r}N^p$ together with
\begin{equation}\label{eq:model}
    (A^px+B^p)N^r=Ax^{p^\ell}+B,
    \qquad \ell=2,3.
\end{equation}
The remaining task is to prove that $N=H^v$ for some $H\in\F_q[x]$, where $v=(p-1)/r$. For these values of $\ell$, the available degree bounds allow one to apply the polynomial $ABC$ theorem to obtain $N=S^vN_1^p$; degree and arithmetic restrictions then force $N_1$ to be constant.

More generally, MVSPs $F\in\F_q[x]$ with value set $V_F=\mathcal C_r\cup\{0\}$ lead to equations of the form
\begin{equation}\label{eq:model*}
    (A^sx+B^s)N^r=Ax^{s^\ell}+B,
    \qquad s=p^e,\quad r\mid s-1.
\end{equation}
For arbitrary $\ell$ and $q$, two difficulties arise. The polynomial $ABC$ theorem alone may not yield the initial factorization $N=S^vN_1^s$, where $v=(s-1)/r$. Moreover, even when this factorization is available, $N_1$ need not be constant. Thus we need both a way to obtain the factorization and a framework in which the argument can be repeated for $N_1$.

Suppose that $N=S^vN_1^s$, with $S$ squarefree. Since $vr=s-1$, multiplying Eq.~\eqref{eq:model*} by $S$ gives
$$
    \big((SA)^sx+(SB)^s\big)N_1^{sr}=SAx^{s^\ell}+SB.
$$
Under suitable degree bounds, comparison of exponents modulo $s$ shows that $SA=A_0^s+xA_1^s$ and $SB=B_0^s+xB_1^s$. Substituting these expressions, separating the terms with exponents congruent to $0$ and $1$ modulo $s$, and taking $s$-th roots, we obtain
\[
\begin{aligned}
    (B_0^s+xB_1^s)N_1^r&=A_0x^{s^{\ell-1}}+B_0,\\
    (A_0^s+xA_1^s)N_1^r&=A_1x^{s^{\ell-1}}+B_1.
\end{aligned}
\]
Thus one equation produces two coupled equations for $N_1$, with $\ell$ replaced by $\ell-1$. Unlike Eq.~\eqref{eq:model*}, each new equation uses different coefficient pairs on its two sides, but all coefficients belong to the common family $\{A_0,A_1,B_0,B_1\}$. The structure to preserve under descent is therefore that of a family of equations.

Our goal is to formalize this procedure in a setting that allows iteration. Hasse derivatives of suitable rational functions yield congruence restrictions on root multiplicities, giving $N=S^vN_1^s$ with $S$ squarefree. We then show that the resulting family of equations for $N_1$ satisfies the same structural and degree conditions. Iteration forces the residual factor to become constant, yielding $N=cH^v$.

Hereafter, we use the following
hypotheses and notation:
\begin{itemize}
    \item $s=p^e>2$ is a power of a prime $p$;
    \item $r\ge 2$ is a divisor of $s-1$ and $e=\mathrm{ord}_rp$;
    \item $v=\frac{s-1}{r}$.
\end{itemize}

We now introduce the families of equations used in the descent.
\begin{definition}\label{def:type2}
Let $\ell$ be a positive integer and  $N\in\F_q[x]$ be nonzero. A {\em covering system} for $(N,r)$ of {\em depth} $\ell$ is a family $\mathcal F=\{P_1,Q_1,\ldots,P_t,Q_t\}$ of polynomials over $\F_q$ satisfying the following conditions:
\begin{enumerate}[(i)]
    \item For each  $a$ with $1\le a\le t$, there is a polynomial identity
    $$
        E_a:\; (Y_a^s+xZ_a^s)N^r
        =P_ax^{s^\ell}+Q_a,
    $$
    where $Y_a,Z_a\in\mathcal F$.

    \item For every root $\lambda$ of $N$, there is an index $a$ with $1\le a\le t$ such that $P_a(\lambda)\ne0$.

    \item If $d$ is the largest degree among the polynomials in $\mathcal F$ and $n=\deg N$, then
    $$
        d+rn<s^\ell.
    $$
\end{enumerate}
\end{definition}
Our goal is to prove the following statement.

\begin{center}
    {\em If $(N,r)$ admits a covering system of depth $\ell\ge1$, then $N=cH^v$ for some $c\in\F_q^\times$ and $H\in\F_q[x]$.}
\end{center}

The proof has two steps: assuming that $(N,r)$ admits a covering system, we first obtain a factorization $N=S^vN_1^s$ with $S$ squarefree, and then show that $(N_1,r)$ admits a covering system of depth $\ell-1$ when $\ell>1$. The following lemma establishes the factorization.

\begin{lemma}\label{lem:0-1}
Let $N\in\F_q[x]$ be nonconstant and suppose that $(N,r)$ admits a covering system of depth $\ell\ge1$. Then every root of $N$ has multiplicity congruent to $0$ or $v$ modulo $s$. Equivalently, $N=S^vN_1^s$ for some $S,N_1\in\F_q[x]$ with $S$ monic and squarefree.
\end{lemma}

\begin{proof}
Let $\mathcal F=\{P_1,Q_1,\ldots,P_t,Q_t\}$ be a covering system for $(N,r)$ of depth $\ell\ge1$. For each index $a$ with $1\le a\le t$, we have
\begin{equation}\label{eq:Ea}
    E_a:\; (Y_a^s+xZ_a^s)N^r=P_ax^{s^\ell}+Q_a,
\end{equation}
where $Y_a,Z_a\in\mathcal F$. Let $n=\deg N$ and let $d$ be the largest degree among the polynomials in $\mathcal F$. By definition, $d+rn<s^\ell$.

By Remark~\ref{rem:01}(ii), $\HD{k}(Y_a^s)=\HD{k}(xZ_a^s)=0$ for $2\le k<s$. Dividing Eq.~\eqref{eq:Ea} by $N^r$ and applying the $k$-th Hasse derivative, we obtain, by Remark~\ref{rem:01}(i) and the product rule~\eqref{eq:hasse-product},
\begin{equation}\label{eq:comb}
    \HD{k}(P_a/N^r)x^{s^\ell}+\HD{k}(Q_a/N^r)=0,
    \qquad 2\le k<s.
\end{equation}

Set $R=\mathrm{rad}(N)$ and $\rho=\deg R$. For $Z=P_a,Q_a$, Lemma~\ref{lem:hasse} shows that $H_k(Z):=N^rR^k\HD{k}(Z/N^r)$ is a polynomial of degree at most $d+k(\rho-1)$. Eq.~\eqref{eq:comb} therefore gives
\begin{equation}\label{eq:comb2}
    H_k(P_a)x^{s^\ell}+H_k(Q_a)=0,
    \qquad 2\le k<s.
\end{equation}

Let $\lambda$ be a root of $N$ of multiplicity $h$, and let $K$ be an integer with $2\le K<s$. Suppose that $d+k(\rho-1)<s^\ell$ for $2\le k\le K$. By the covering condition, we may choose an index $b$ such that $P_b(\lambda)\ne0$. By the degree bound above, the two terms in Eq.~\eqref{eq:comb2} have disjoint degree ranges, so $H_k(P_b)=H_k(Q_b)=0$. Hence $\HD{k}(P_b/N^r)=0$ for $2\le k\le K$. The function $P_b/N^r$ has a pole of order $rh$ at $\lambda$, so Remark~\ref{rem:vanish} gives $\binom{-rh}{k}\equiv0\pmod p$ for $2\le k\le K$.

We now prove the required restrictions on root multiplicities.
\begin{enumerate}[(a)]
\item Every root of $N$ has multiplicity at least $v$.

Let $j$ be the smallest multiplicity among the roots of $N$, and suppose that $j<v$. Then $j\rho\le n$ and $M:=rj$ satisfies $2\le M<s$. For $2\le k\le M$, we have
$$
    d+k(\rho-1)<d+rj\rho\le d+rn<s^\ell.
$$
The preceding argument with $h=j$ and $K=M$ gives $\binom{-M}{k}\equiv0\pmod p$ for $2\le k\le M$. By Lemma~\ref{lem:binomial}(iii), $M=p^u-1$ for some integer $u\ge1$. Since $r\mid M$ and $e=\mathrm{ord}_rp$, we have $e\mid u$. Hence $rj=M\ge p^e-1=rv$, contradicting $j<v$.

\item Every root of $N$ has multiplicity congruent to $0$ or $v$ modulo $s$.

By part~(a), $\rho\le n/v$. For $2\le k<s$, we therefore have
$$
    d+k(\rho-1)<d+(s-1)\rho\le d+rn<s^\ell.
$$
For any root of multiplicity $h$, the same argument with $K=s-1$ gives $\binom{-rh}{k}\equiv0\pmod p$ for $2\le k<s$. By Lemma~\ref{lem:binomial}(ii), $-rh\equiv0$ or $1\pmod s$. Since $rv=s-1$ and $\gcd(r,s)=1$, we conclude that $h\equiv0$ or $v\pmod s$.
\end{enumerate}

Taking $S$ as the monic squarefree product of the irreducible factors of $N$ with multiplicity congruent to $v$ modulo $s$ gives $N=S^vN_1^s$.
\end{proof}

The previous lemma enables us to recursively construct covering systems.

\begin{proposition}\label{prop:recursive}
     Let $N=S^{v}N_1^{s}$ be a nonzero polynomial over $\F_q$ such that $S$ is monic squarefree. If ${\ell}>1$ and $(N, r)$ admits a covering system of depth $\ell$, then $(N_1, r)$ admits a covering system of depth $\ell-1$.
\end{proposition}
\begin{proof}
Let $\mathcal F=\{P_1,Q_1,\ldots,P_t,Q_t\}$ be a covering system for $(N,r)$ of depth $\ell$. We construct a covering system for $(N_1,r)$ of depth $\ell-1$.

For each index $a$ with $1\le a\le t$, we have
\begin{equation}\label{eq:Eb}
    (Y_a^s+xZ_a^s)N^r=P_ax^{s^\ell}+Q_a,
\end{equation}
where $Y_a,Z_a\in\mathcal F$. Let $n=\deg N$, $\delta=\deg S$, and let $d$ be the largest degree among the polynomials in $\mathcal F$. Then $d+rn<s^\ell$.

Multiplying Eq.~\eqref{eq:Eb} by $S$ and using $N=S^vN_1^s$ and $vr+1=s$, we obtain
\begin{equation}\label{eq:ps}
    \big((SY_a)^s+x(SZ_a)^s\big)N_1^{rs}
    =SP_ax^{s^\ell}+SQ_a.
\end{equation}
Every monomial on the left has degree congruent to $0$ or $1$ modulo $s$. Since $\deg(SP_a),\deg(SQ_a)\le d+\delta\le d+rn<s^\ell$, the two terms on the right have disjoint degree ranges. Hence every monomial of $SP_a$ and $SQ_a$ has degree congruent to $0$ or $1$ modulo $s$. Thus, for every $W\in\mathcal F$, there are unique polynomials $W_0,W_1\in\F_q[x]$ such that $SW=W_0^s+xW_1^s$.

Substituting these decompositions into Eq.~\eqref{eq:ps}, separating the exponents congruent to $0$ and $1$ modulo $s$, and taking $s$-th roots, we obtain
\[
\begin{aligned}
    (Y_{a,0}^s+xY_{a,1}^s)N_1^r
        &=P_{a,0}x^{s^{\ell-1}}+Q_{a,0},\\
    (Z_{a,0}^s+xZ_{a,1}^s)N_1^r
        &=P_{a,1}x^{s^{\ell-1}}+Q_{a,1}.
\end{aligned}
\]
We claim that $\mathcal G:=\{P_{a,0},Q_{a,0},P_{a,1},Q_{a,1}:1\le a\le t\}$ is a covering system for $(N_1,r)$ of depth $\ell-1$.

\begin{enumerate}[(i)]
\item The displayed equations have the required form. Since $Y_a,Z_a\in\mathcal F$, the polynomials $Y_{a,0},Y_{a,1},Z_{a,0},Z_{a,1}$ belong to $\mathcal G$.

\item Let $\lambda$ be a root of $N_1$. Then $\lambda$ is also a root of $N$, so there is an index $b$ such that $P_b(\lambda)\ne0$. If $P_{b,0}(\lambda)=P_{b,1}(\lambda)=0$, the identity $SP_b=P_{b,0}^s+xP_{b,1}^s$ shows that $SP_b$ vanishes at $\lambda$ with multiplicity at least $s$. This is impossible because $S$ is squarefree and $P_b(\lambda)\ne0$. Hence at least one of $P_{b,0},P_{b,1}$ is nonzero at $\lambda$.

\item Let $d_1$ be the largest degree among the polynomials in $\mathcal G$ and let $n_1=\deg N_1$. The terms in $SW=W_0^s+xW_1^s$ have distinct degree residues modulo $s$, so $sd_1\le d+\delta$. Also, $n=v\delta+sn_1$. Therefore,
$$
    d_1+rn_1
    \le\frac{d+\delta+r(n-v\delta)}s
    =\frac{d+rn-(s-2)\delta}s
    \le\frac{d+rn}s
    <s^{\ell-1}.
$$
\end{enumerate}

Thus $\mathcal G$ is a covering system for $(N_1,r)$ of depth $\ell-1$.
\end{proof}

Combining the previous results, we obtain the following theorem.

 \begin{theorem}\label{thm:covering}
Let $N\in \F_q[x]$ be nonzero and suppose that $(N, r)$ admits a covering system of depth $\ell\ge 1$. Then there exist $H\in \F_q[x]$ and $c\in \F_q$ such that
 $$N=cH^{v}.$$
 \end{theorem}

\begin{proof}
We argue by induction on $\ell$. The result is immediate if $N$ is constant, so assume otherwise. By Lemma~\ref{lem:0-1}, write $N=S^vN_1^s$, with $S$ monic and squarefree.

If $\ell=1$, the degree condition gives $\deg N<s$, so $\deg N_1\le\deg N/s<1$. Thus $N_1$ is constant, and the result follows.

If $\ell>1$, Proposition~\ref{prop:recursive} gives a covering system for $(N_1,r)$ of depth $\ell-1$. By induction, $N_1=cH^v$ for some $c\in\F_q^\times$ and $H\in\F_q[x]$. Hence $N=c^s(SH^s)^v$, as required.
\end{proof}

\subsection{On general covering systems}
In this subsection, we seek to extend the concept of covering systems to a new setting, which will be particularly relevant when treating case (i) of Theorem~\ref{thm:conjecture}.

Hereafter, we fix the following
hypotheses and notation:
\begin{itemize}
    \item $s=p^e>2$ is a power of a prime $p$;
    \item for each integer $i\ge1$, let $s_i=\frac{s^i-1}{s-1}$ and 
$$\mathcal D_i:=\left\{\sum_{j=0}^{i-1}\varepsilon_js^j:\varepsilon_j\in \{0, 1\}\right\};$$
\item for each integer $i\ge1$ and each nonempty family $\mathcal F$ of polynomials over $\F_q$, let
$$
    \mathcal F^{(i)}
    =\langle\{x^gf^{s^i}:g\in\mathcal D_i,\ f\in\mathcal F\}\rangle_{\F_q}
$$
be the $\F_q$-vector space generated by these polynomials.
\end{itemize}

We introduce the following definition.

\begin{definition}\label{def:type2*}
Let $\ell$ be a positive integer and $N\in\F_q[x]$ be nonzero. A {\em general covering system} for $N$ of {\em depth} $\ell$ is a family $\mathcal F=\{P_1,Q_1,\ldots,P_t,Q_t\}$ of polynomials over $\F_q$ satisfying the following conditions:
\begin{enumerate}[(i)]
    \item For each index $a$ with $1\le a\le t$, there is an identity
$$
    E_a:\; P_ax^{s^\ell}+Q_a+\sum_{i\ge1}Z_{a,i}N^{s_i}=0,
$$
where the sum is finite and $Z_{a,i}\in\mathcal F^{(i)}$.
\item For every root $\lambda$ of $N$, there is an index $a$ with $1\le a\le t$ such that $P_a(\lambda)\ne0$;
\item Let $d$ be the maximum of the degrees of the polynomials in $\mathcal F$ and $n=\deg N$. Then
$$
    d+(s-1)n<s^\ell.
$$
\end{enumerate}
\end{definition}

\begin{remark}\label{rem:03}
For every $g\in\mathcal D_i$, write $g=\varepsilon+sg'$, where $\varepsilon\in\{0,1\}$ and $g'\ge0$ is an integer. Then $g',g'+s^{i-1}\in\mathcal D_i$. Moreover, for every $V\in\F_q(x)$, we have $x^gV^s=A^sx+B^s$ for some $A,B\in\F_q(x)$. By Remark~\ref{rem:01}(ii),
$$
    \HD{k}(x^gV^s)=0,\qquad 2\le k<s.
$$
\end{remark}

We will prove that every polynomial admitting a general covering system has the form $cH^{s-1}$. We begin with an analogue of Lemma~\ref{lem:0-1}.

\begin{lemma}\label{lem:0-1*}
   Let $N\in\F_q[x]$ be a nonconstant polynomial that admits a general covering system of depth $\ell\ge1$. Then every root $\lambda$ of $N$ has multiplicity congruent to $0$ or $-1$ modulo $s$. Equivalently, $N=S^{s-1}N_1^{s}$ for some $S, N_1\in \F_q[x]$ with $S$ monic and squarefree.   
\end{lemma}

\begin{proof}
Let $\mathcal F=\{P_1, Q_1, \ldots, P_t, Q_t\}$ be a general covering system for $N$ of depth ${\ell}\ge 1$.
From hypothesis, each index $a$ with $1\le a\le t$ gives rise to an equation
\begin{equation}\label{eq:Ea*}
    E_a: \; P_ax^{s^{\ell}}+Q_a+\sum_{i\ge 1}Z_{a, i}N^{s_i}=0,
    \end{equation}
where the sum is finite and $Z_{a, i}\in \mathcal F^{(i)}$. Moreover, if $n=\deg N$ and $d$ is the maximum of the degrees of the polynomials in $\mathcal F$, we have $d+(s-1)n<s^{\ell}$.

Each element of $\mathcal F^{(i)}$ is an $\F_q$-linear combination of polynomials of the form $x^gh^s$ with $g\in\mathcal D_i$ and $h\in\F_q[x]$. By Remark~\ref{rem:03}, every $Z\in\mathcal F^{(i)}$ satisfies $\HD{k}(Z)=0$ for $2\le k<s$.
Dividing Eq.~\eqref{eq:Ea*} by $N$ and taking the $k$-th Hasse derivative for $2\le k<s$, we obtain
\begin{equation}\label{eq:comb*}
    \HD{k}(P_a/N)x^{s^\ell}+\HD{k}(Q_a/N)=0.
\end{equation}
Indeed, since $s_i-1$ is divisible by $s$, Remark~\ref{rem:01}(i) and the product rule \eqref{eq:hasse-product} give $\HD{k}(Z_{a,i}N^{s_i-1})=N^{s_i-1}\HD{k}(Z_{a,i})=0$.

Set $R=\mathrm{rad}(N)$ and $\rho=\deg R$. 
For $Z=P_a,Q_a$ and $2\le k<s$, Lemma~\ref{lem:hasse} shows that $H_k(Z)=NR^k\HD{k}(Z/N)$ is a polynomial satisfying
\begin{equation}\label{eq:H_kine*}
    \deg H_k(Z)\le d+k(\rho-1)<d+(s-1)\rho\le d+(s-1)n<s^\ell.
\end{equation} 

Multiplying Eq.~\eqref{eq:comb*} by $NR^k$ gives
\begin{equation}\label{eq:comb2*}
    H_k(P_a)x^{s^\ell}+H_k(Q_a)=0,\qquad 2\le k<s.
\end{equation}
By Eq.~\eqref{eq:H_kine*}, the two terms have disjoint degree ranges. Hence $H_k(P_a)=H_k(Q_a)=0$, so $\HD{k}(P_a/N)=\HD{k}(Q_a/N)=0$ for $2\le k<s$.

Let $\lambda$ be a root of $N$ with multiplicity $j\ge1$. By condition (ii) in Definition~\ref{def:type2*}, there is an index $a$ with $1\le a\le t$ such that $P_a(\lambda)\ne0$. In particular, $P_a/N$ has a pole of order $j$ at $\lambda$.
Since $\HD{k}(P_a/N)=0$ for every integer $k$ with $2\le k<s$, Remark~\ref{rem:vanish} implies that $\binom{-j}{k}\equiv0\pmod p$ for every such $k$. Together with Lemma~\ref{lem:binomial}(ii), this gives $-j\equiv0,1\pmod s$, and the result follows.
\end{proof}

We obtain the following theorem.

\begin{theorem}\label{thm:covering*}
    Suppose $N\in \F_q[x]$ admits a general covering system of depth $\ell\ge 1$. Then $N=cH^{s-1}$ for some $c\in \F_q$ and $H\in \F_q[x]$.  
\end{theorem}

\begin{proof}
We argue by induction on $\ell$. The result is trivial if $N$ is constant. Assuming otherwise, Lemma~\ref{lem:0-1*} gives $N=S^{s-1}M^s$ for some $S,M\in\F_q[x]$ with $S$ squarefree. If $\ell=1$, condition (iii) in Definition~\ref{def:type2*} gives $\deg N<s$, so $M$ is constant and the result follows. For $\ell>1$, it suffices to prove that $M$ admits a general covering system of depth $\ell-1$.

Let $\mathcal F=\{P_1,Q_1,\ldots,P_t,Q_t\}$ be a general covering system for $N$ of depth $\ell\ge1$.
By hypothesis, each index $a$ with $1\le a\le t$ gives rise to an equation
\begin{equation}\label{eq:Eb*}
    E_a:\; P_ax^{s^\ell}+Q_a+\sum_{i\ge1}Z_{a,i}N^{s_i}=0,
\end{equation}
where the sum is finite and $Z_{a,i}\in\mathcal F^{(i)}$. Moreover, if $n=\deg N$ and $d$ is the maximum of the degrees of the polynomials in $\mathcal F$, we have $d+(s-1)n<s^\ell$. Multiplying both sides of Eq.~\eqref{eq:Eb*} by $S$, we obtain
$$
    P_aSx^{s^\ell}+Q_aS+\sum_{i\ge1}Z_{a,i}SN^{s_i}=0.
$$
Since $(s-1)s_i+1=s^i$ and $N=S^{s-1}M^s$, we obtain $SN^{s_i}=S^{s^i}M^{s_is}$, hence 
\begin{equation}\label{eq:ps*}
    P_aSx^{s^{\ell}}+Q_aS+\sum_{i\ge 1}Z_{a, i}S^{s^i}M^{s_is}=0.
\end{equation}

Since every $g\in\mathcal D_i$ satisfies $g\equiv0,1\pmod s$, it follows by the definition of $\mathcal F^{(i)}$ that each monomial in $Z_{a,i}S^{s^i}M^{s_is}$ has degree congruent to $0$ or $1$ modulo $s$. By Eq.~\eqref{eq:ps*}, the same holds for each monomial in $P_aSx^{s^\ell}+Q_aS$.
Since $\deg S\le n$ and $s>2$, we obtain
$$
    \max\{\deg(P_aS),\deg(Q_aS)\}
    \le d+\deg S\le d+(s-1)n<s^\ell.
$$
In particular, there is no monomial cancellation in the sum $P_aSx^{s^\ell}+Q_aS$.
Thus each monomial in $P_aS$ and $Q_aS$ has degree congruent to $0$ or $1$ modulo $s$. Hence there exist polynomials $P_{a,0},P_{a,1},Q_{a,0},Q_{a,1}\in\F_q[x]$ such that
$$
P_aS=P_{a,1}^sx+P_{a,0}^s
\quad\text{and}\quad
Q_aS=Q_{a,1}^sx+Q_{a,0}^s.
$$

These expressions hold for every polynomial in $\mathcal F$. By hypothesis, each $Z_{a,i}$ in Eq.~\eqref{eq:ps*} belongs to $\mathcal F^{(i)}$. Thus
$$
Z_{a,i}=\sum_{g\in\mathcal D_i}\sum_{W\in\mathcal F}
c_{a,g,i,W}x^gW^{s^i},
$$
where $c_{a,g,i,W}\in\F_q$, and hence
$$
Z_{a,i}S^{s^i}=\sum_{g\in\mathcal D_i}\sum_{W\in\mathcal F}
c_{a,g,i,W}x^g(WS)^{s^i}.
$$
Since $W\in\mathcal F$, we have $WS=W_1^sx+W_0^s$ for some $W_0,W_1\in\F_q[x]$. Thus
$$
Z_{a,i}S^{s^i}
=\sum_{g\in\mathcal D_i}\sum_{W\in\mathcal F}
c_{a,g,i,W}\left(
x^{g+s^i}W_1^{s^{i+1}}+x^gW_0^{s^{i+1}}
\right).
$$
Writing $g=\alpha_g+sg'$ with $\alpha_g\in\{0,1\}$, we can rewrite this as
$$
Z_{a,i}S^{s^i}
=\left(Z_{a,i}^{(1)}\right)^sx
+\left(Z_{a,i}^{(0)}\right)^s,
$$
where, for each $\varepsilon\in\{0,1\}$,
$$
Z_{a,i}^{(\varepsilon)}
:=\sum_{g\in\mathcal D_i\atop \alpha_g=\varepsilon}
\sum_{W\in\mathcal F}
c_{a,g,i,W}^{1/s}
\left(x^{g'+s^{i-1}}W_1^{s^i}+x^{g'}W_0^{s^i}\right).
$$
Therefore, Eq.~\eqref{eq:ps*} becomes
$$
\left(P_{a,1}x^{s^{\ell-1}}+Q_{a,1}
+\sum_{i\ge1}Z_{a,i}^{(1)}M^{s_i}\right)^sx
+\left(P_{a,0}x^{s^{\ell-1}}+Q_{a,0}
+\sum_{i\ge1}Z_{a,i}^{(0)}M^{s_i}\right)^s=0.
$$
The two terms have exponents congruent to $1$ and $0$ modulo $s$, respectively, so
$$
E_{a,1}:\; P_{a,1}x^{s^{\ell-1}}+Q_{a,1}
+\sum_{i\ge1}Z_{a,i}^{(1)}M^{s_i}=0,
$$
$$
E_{a,0}:\; P_{a,0}x^{s^{\ell-1}}+Q_{a,0}
+\sum_{i\ge1}Z_{a,i}^{(0)}M^{s_i}=0.
$$
We claim that
$$
\mathcal G:=\{P_{a,0},P_{a,1},Q_{a,0},Q_{a,1}:1\le a\le t\}
$$
is a general covering system for $M$ of depth $\ell-1$. We verify the three defining conditions.

\begin{enumerate}[(i)]
    \item We use the equations $E_{a,1}$ and $E_{a,0}$ for each index $a$ with $1\le a\le t$. Since each $W\in\mathcal F$ is one of the polynomials $P_a,Q_a$, both $W_0$ and $W_1$ belong to $\mathcal G$. By Remark~\ref{rem:03}, we have $g',g'+s^{i-1}\in\mathcal D_i$ whenever $g\in\mathcal D_i$ and $g=\alpha_g+sg'$ with $\alpha_g\in\{0,1\}$. Thus $Z_{a,i}^{(0)},Z_{a,i}^{(1)}\in\mathcal G^{(i)}$ for each index $a$ with $1\le a\le t$ and each integer $i\ge1$, proving the first condition.

    \item Let $\lambda$ be a root of $M$. Since $\lambda$ is also a root of $N$, condition (ii) in Definition~\ref{def:type2*} gives an index $b$ with $1\le b\le t$ such that $P_b(\lambda)\ne0$. If $P_{b,1}(\lambda)=P_{b,0}(\lambda)=0$, the identity
$$
P_bS=P_{b,1}^sx+P_{b,0}^s
$$
shows that $P_bS$ vanishes at $\lambda$ with multiplicity at least $s$. This contradicts $P_b(\lambda)\ne0$, since $S$ is squarefree. Hence $P_{b,1}(\lambda)\ne0$ or $P_{b,0}(\lambda)\ne0$, proving the second condition.

    \item Finally, let $d_1$ be the maximum of the degrees of the polynomials in $\mathcal G$, and set $m=\deg M$ and $\delta=\deg S$. Since the two terms in $P_{a,1}^sx+P_{a,0}^s$ have exponents in distinct residue classes modulo $s$, the identity $SP_a=P_{a,1}^sx+P_{a,0}^s$ gives $s\deg P_{a,j}\le\delta+d$ for each $j\in\{0,1\}$. The same argument applies to $Q_{a,0},Q_{a,1}$; hence $d_1\le(\delta+d)/s$. Since $N=S^{s-1}M^s$, we have $m=(n-(s-1)\delta)/s$. Thus
$$
d_1+(s-1)m
\le\frac{\delta+d+(s-1)(n-(s-1)\delta)}s
\le\frac{d+(s-1)n}s
<s^{\ell-1}.
$$
This proves the last condition.
\end{enumerate}
By induction, $M=cH^{s-1}$ for some $c\in\F_q$ and $H\in\F_q[x]$. Hence $N=S^{s-1}M^s=c^s(SH^s)^{s-1}$, completing the proof.
\end{proof}

\section{Proof of the main results}\label{sec:proof}
We split the proof of Theorem~\ref{thm:conjecture} into two cases according to whether $\mathcal U$ is a field.

\subsection{The case $\mathcal U$ is a field}
Let $F\in\mathcal P(\mathcal U^v,q)$. Then $\mathcal U^v=\mathcal C_r\cup\{0\}$, where $\mathcal C_r\subseteq\F_q$ is the cyclic group of order $r=(p^{mk}-1)/v$.

By Remark~\ref{rem:field}, we may assume that $\mathcal U^v$ is not a field, so $v>1$. Let $\gamma_0\in\F_q$ be the element associated to $F$ in Theorem~\ref{thm:mills}. By Lemma~\ref{lem:gamma}, we have $\gamma_0=0$. Therefore, Theorem~\ref{thm:mills} gives integers $m_0,k_0,v_0$ and polynomials $A,B,L,N\in\F_q[x]$ satisfying the following:
\begin{itemize}
    \item $v_0$ divides $p^{k_0}-1$ and $\frac{p^{m_0k_0}-1}{v_0}=r$;
    \item $L=\gcd(F, x^q-x)=A^{p^{m_0k_0}}x+B^p$;
    \item $F=L^{v_0}N^{p^{m_0k_0}}$.
\end{itemize}
Let $e$ be the smallest positive integer such that $\mathcal U^v=\mathcal C_r\cup\{0\}\subseteq\F_{p^e}$, that is, $e=\mathrm{ord}_rp$. Since $r\mid p^{m_0k_0}-1$, we have $m_0k_0=e\ell$ for some integer $\ell\ge1$.
In particular, for $s=p^e$ and $t=(s-1)/r$, we have
$$
r=\frac{s-1}{t}=\frac{p^{m_0k_0}-1}{v_0},
\qquad\text{and}\qquad
p^{m_0k_0}=s^\ell.
$$
Therefore, $t\equiv v_0\pmod s$ and $v_0\ge t$, so we may write
$$
F=L^tN_1^s,
\qquad N_1=L^{(v_0-t)/s}N^{s^{\ell-1}}.
$$
Since $V_F=\mathcal C_r\cup\{0\}$, the polynomial $T$ in Theorem~\ref{MillsBor} is $T(x)=x^{r+1}-x$, and $T'(a)=r$ for each $a\in V_F\setminus\{0\}$. Thus $\theta=-r$. Since $F=L^tN_1^s$ and $tr\equiv-1\pmod p$, Eq.~\eqref{MVSP} gives
\begin{equation}\label{eq:F-differential}
(L^tN_1^s)^{r+1}-L^tN_1^s
=(x^q-x)L^{t-1}L'N_1^s.
\end{equation}
Set $A_1=A^{s^{\ell-1}}$, so $L'=A^{p^{m_0k_0}}=A_1^s$. Since $tr+1=s$, Eq.~\eqref{eq:F-differential} simplifies to
$$
L^sN_1^{sr}-L=(x^q-x)A_1^s.
$$
Substituting $L=xA_1^s+B^p$, we obtain
\begin{equation}\label{eq:r-1}
(LN_1^r)^s-B^p=x^qA_1^s.
\end{equation}
Since $r\mid q-1$ and $e=\mathrm{ord}_rp$, we may write $q=s^h$ for some integer $h\ge1$. Thus $x^q$ is an $s$-th power, and Eq.~\eqref{eq:r-1} gives $B^p=B_1^s$ for some $B_1\in\F_q[x]$. Taking $s$-th roots, we obtain $LN_1^r=A_1x^{s^{h-1}}+B_1$. Since $L=xA_1^s+B_1^s$, this becomes
\begin{equation}\label{eq:r-2}
(B_1^s+xA_1^s)N_1^r=A_1x^{s^{h-1}}+B_1.
\end{equation}
If $h=1$, the MVSP degree bound gives $\deg F\le(q-1)/r<s$. Since $F=L^tN_1^s$, it follows that $N_1$ is constant.

For $h\ge2$, we use the following lemma. Recall that $s>2$.

\begin{lemma}\label{lem:field-covering}
Suppose $h\ge2$. The family $\mathcal F=\{A_1,B_1\}$ is a covering system for $(N_1,r)$ of depth $h-1$.
\end{lemma}

\begin{proof}
 
Condition (i) in Definition~\ref{def:type2} follows from Eq.~\eqref{eq:r-2}. To verify condition (ii), let $\lambda$ be a root of $N_1$. If $A_1(\lambda)=0$, Eq.~\eqref{eq:r-2} gives $B_1(\lambda)=0$. Then $L=xA_1^s+B_1^s$ vanishes at $\lambda$ with multiplicity at least $s$, contradicting the fact that $L$ is squarefree. Hence $A_1(\lambda)\ne0$. It remains to verify condition (iii) in Definition~\ref{def:type2}. Let $d=\max\{\deg A_1,\deg B_1\}$ and $n=\deg N_1$. We must show that $d+rn<s^{h-1}$.
If $\deg B_1>\deg A_1$, the left-hand side of Eq.~\eqref{eq:r-2} has degree $sd+rn>\deg B_1$. Comparing degrees in Eq.~\eqref{eq:r-2}, we obtain
$$
sd+rn=\deg A_1+s^{h-1}<d+s^{h-1}.
$$
Since $s>2$, it follows that $d+rn<s^{h-1}$.
If $\deg B_1\le\deg A_1$, the left-hand side of Eq.~\eqref{eq:r-2} has degree $sd+1+rn>\deg B_1$. Comparing degrees in Eq.~\eqref{eq:r-2}, we obtain
$$
sd+1+rn=d+s^{h-1},
$$
so $d+rn<s^{h-1}$. This completes the proof.
\end{proof}

For $h\ge2$, Lemma~\ref{lem:field-covering} and Theorem~\ref{thm:covering} give $N_1=cH^t$ for some $c\in\F_q$ and $H\in\F_q[x]$. The same conclusion holds for $h=1$, since $N_1$ is constant. Thus, setting $\alpha=c^s$, we obtain
$$
F=\alpha L^tH^{ts}=\alpha(LH^s)^t.
$$
Recall that $F\in\mathcal P(\mathcal U^v,q)=\mathcal P(\mathcal C_r\cup\{0\},q)$. Since $rt=s-1\mid q-1$, the cyclic group $\F_q^\times$ satisfies $\mathcal C_r\subseteq(\F_q^\times)^t$. Choose $u\in\F_q$ with $F(u)\ne0$. Then $F(u)\in\mathcal C_r$ and $F(u)=\alpha((LH^s)(u))^t$, so $\alpha=a^t$ for some $a\in\F_q^\times$. Hence $F=H_0^t$, where $H_0=aLH^s$. Since $\mathcal C_r\cup\{0\}=(\F_s)^t$, we have $H_0^t\in\mathcal P((\F_s)^t,q)$. Lemma~\ref{lem:br-v} then gives $H_0\in\mathcal P(\F_s,q)$, completing this case.

\subsection{The case $\mathcal U$ is not a field}

The result is trivial if $v=1$, so we assume that $v>1$. Set $r=(p^{mk}-1)/v$. By Remark~\ref{rem:field}, we may assume that $\mathcal U^v$ is not a field.

Fix $F\in\mathcal P(\mathcal U^v,q)$ and let $\gamma_0\in\F_q$ be the element associated to $F$ in Theorem~\ref{thm:mills}. By Lemma~\ref{lem:gamma}, we have $\gamma_0=0$. Therefore, Theorem~\ref{thm:mills} gives integers $m_0,k_0,v_0$ and polynomials $A,B,L,N\in\F_q[x]$ satisfying the following:
\begin{itemize}
    \item $v_0$ divides $p^{k_0}-1$ and $\frac{p^{m_0k_0}-1}{v_0}=r$;
    \item $L=\gcd(F, x^q-x)=A^{p^{m_0k_0}}x+B^p$;
    \item $F=L^{v_0}N^{p^{m_0k_0}}$.
\end{itemize}
Under these conditions, the proof of Theorem~2.6 in~\cite{BR} shows that $m_0k_0=mk$ and $v_0=v$. Since $\mathcal U$ is not a field and $1\in\mathcal U$, we have $m>1$. Hence $\#\mathcal U^{p^k-1}=(p^{mk}-1)/(p^k-1)+1>2$. By Lemma~\ref{lem:reduction}, it therefore suffices to consider the case $v=p^k-1$.
Set $s=p^k$, so $v=s-1$, and write
$$
L_{\mathcal U}(x):=\prod_{u\in\mathcal U}(x-u)
=\sum_{i=0}^m a_ix^{s^i},
$$
where $a_0,a_m\ne0$.
Let $T\in\F_q[x]$ be the polynomial associated to $F$ in Theorem~\ref{MillsBor}. Then
$$
\frac{T(x^{s-1})}{x^{s-2}}=L_{\mathcal U}(x),
\qquad\text{and hence}\qquad
T(x)=\sum_{i=0}^m a_ix^{s_i+1},
$$
where $s_i=(s^i-1)/(s-1)$.
Theorem~\ref{MillsBor} gives
$$
T(F)=\theta(x^q-x)F'
$$
for some $\theta\in\F_q^\times$. Substituting $F=L^{s-1}N^{s^m}$ and $L'=A^{s^m}$ into Eq.~\eqref{MVSP}, we obtain
$$
\sum_{i=0}^m a_iL^{(s-1)(s_i+1)}N^{s^m(s_i+1)}
=-\theta(x^q-x)L^{s-2}A^{s^m}N^{s^m}.
$$
Set $Q=s^m$. Since $(s-1)s_i+1=s^i$, Eq.~\eqref{MVSP} yields
\begin{equation}\label{eq:mill-U}
a_0L+\sum_{i=1}^m a_iL^{s^i}N^{s_iQ}
=-\theta(x^q-x)A^Q.
\end{equation}
Differentiating Eq.~\eqref{eq:mill-U}, we obtain $a_0A^Q=\theta A^Q$. Since $L$ is squarefree, $L'=A^Q\ne0$, so $\theta=a_0$. Dividing Eq.~\eqref{eq:mill-U} by $a_0$, we obtain
\begin{equation}\label{eq:mill-U-2}
L+\sum_{i=1}^m b_iL^{s^i}N^{s_iQ}
=A^Qx-A^Qx^q,
\qquad b_i=a_i/a_0.
\end{equation}

Write $q/Q=s^\ell$. Since $\mathcal U\subseteq\F_q$ is not a field and $\#\mathcal U=Q$, we have $\mathcal U\ne\F_q$, so $\ell\ge1$. Eq.~\eqref{eq:mill-U-2} yields the following lemma.

\begin{lemma}\label{lem:nonfield-covering}
There is a general covering system for $N$ of depth $\ell$.
\end{lemma}

\begin{proof}
Put $R=\F_q[x^Q]$ and let
$$
V=R\oplus\bigoplus_{j=0}^{m-1}x^{s^j}R.
$$
Thus $V$ consists of the polynomials over $\F_q$ whose monomials have degrees congruent to $0$ or $s^j$ modulo $Q$ for some integer $j$ with $0\le j<m$.
Define $\Psi:\F_q[x]\to\F_q[x]$ by
$$
\Psi(Z)=\sum_{i=1}^m b_iN^{Qs_i}Z^{s^i}.
$$
The map $\Psi$ is additive and preserves $V$. Since each application multiplies monomial degrees modulo $Q$ by a positive power of $s$, its $m$-fold iterate satisfies $\Psi^m(\F_q[x])\subseteq R$. Eq.~\eqref{eq:mill-U-2} gives
$$
L+\Psi(L)=A^Qx-A^Qx^q\in V.
$$
By additivity of $\Psi$ and invariance of $V$, we have
$$
L+(-1)^{m-1}\Psi^m(L)
=\sum_{i=0}^{m-1}(-1)^i\Psi^i(L+\Psi(L))
\in V.
$$
Since $\Psi^m(L)\in R\subseteq V$, it follows that $L\in V$.
Hence there exist unique polynomials $C,D_0,\ldots,D_{m-1}\in\F_q[x]$ such that
\begin{equation}\label{eq:L-decom}
L=C^Q+\sum_{j=0}^{m-1}D_j^Qx^{s^j}.
\end{equation}
Set $\beta_i=b_i^{1/Q}$. Comparing residue classes modulo $Q$ in Eq.~\eqref{eq:mill-U-2} and taking $Q$-th roots, we obtain $D_0=A$,
\begin{equation}\label{eq:D-relations}
D_j+\sum_{i=1}^j\beta_iD_{j-i}^{s^i}N^{s_i}=0,
\qquad 1\le j<m,
\end{equation}
and
\begin{equation}\label{eq:admissible}
0=x^{q/Q}A+C+\sum_{i=1}^m\beta_iC^{s^i}N^{s_i}
+\sum_{j=0}^{m-1}\sum_{i=m-j}^m
\beta_i x^{s^{i+j-m}}D_j^{s^i}N^{s_i}.
\end{equation}

Recall that $q/Q=s^\ell$. To construct a general covering system, we take
\[
(P_0,Q_0)=(A,C),
\qquad
(P_j,Q_j)=(0,D_j)\quad (1\le j<m).
\]
Indeed, Eq.~\eqref{eq:admissible} can be written as
$$
x^{q/Q}A+C+\sum_{i\ge1}\Delta_iN^{s_i}=0,
$$
where each $\Delta_i\in\F_q[x]$ is a sum of terms of the form $\alpha x^{g_i}Z_i$, with $\alpha\in\F_q$ and
\begin{itemize}
    \item $g_i=0$ or $g_i=s^{i+j-m}$ with $0\le i+j-m<i$;
    \item $Z_i\in\{P_0^{s^i},Q_0^{s^i},\ldots,P_{m-1}^{s^i},Q_{m-1}^{s^i}\}$.
\end{itemize}
Thus $g_i\in\mathcal D_i$ in both cases, so $\Delta_i\in\mathcal F^{(i)}$, where $\mathcal F=\{P_0,Q_0,\ldots,P_{m-1},Q_{m-1}\}$. The relations in Eq.~\eqref{eq:D-relations} have the same required form, with $P_j=0$ and $Q_j=D_j$. Hence condition (i) in Definition~\ref{def:type2*} holds.

For condition (ii), suppose that $\lambda$ is a common root of $A$ and $N$. Eq.~\eqref{eq:admissible} gives $C(\lambda)=0$, and Eq.~\eqref{eq:D-relations} gives $D_j(\lambda)=0$ for each integer $j$ with $1\le j<m$. Since $D_0=A$, Eq.~\eqref{eq:L-decom} yields $L(\lambda)=0$. Moreover, $L'(\lambda)=A(\lambda)^Q=0$, contradicting the fact that $L=\gcd(F,x^q-x)$ is squarefree. Thus $\gcd(A,N)=1$, so $P_0=A$ is nonzero at every root of $N$.

It remains to verify condition (iii) in Definition~\ref{def:type2*}. Let $a=\deg A$, $n=\deg N$, and $d=\max\{\deg C,\deg D_0,\ldots,\deg D_{m-1}\}$. We must show that $d+(s-1)n<q/Q=s^\ell$.
Since $b_m=a_m/a_0=1/a_0\ne0$ and $L$ is nonconstant, the term with $i=m$ on the left-hand side of Eq.~\eqref{eq:mill-U-2} has degree greater than that of every other term. Comparing degrees in Eq.~\eqref{eq:mill-U-2}, we obtain
\[
\deg L+s_mn=q/Q+a.
\]
Since $m>1$, we have $s_m=(s^m-1)/(s-1)>s-1$. Hence
$$
\deg L+(s-1)n\le q/Q+a.
$$
The summands in Eq.~\eqref{eq:L-decom} have degrees in distinct residue classes modulo $Q$. Hence $\deg L\ge Qd$ and, since $D_0=A$, $\deg L\ge Qa+1$. Together, these bounds imply $\deg L>d+a$, since $Q\ge2$. Therefore,
$$
d+(s-1)n<\deg L-a+(s-1)n\le q/Q=s^\ell.
$$
This proves condition (iii).
\end{proof}
Lemma~\ref{lem:nonfield-covering} and Theorem~\ref{thm:covering*} give $N=cH^{s-1}$ for some $c\in\F_q$ and $H\in\F_q[x]$. Setting $\alpha=c^Q$, we obtain
$$
F=\alpha L^{s-1}H^{Q(s-1)}
=\alpha(LH^Q)^{s-1}.
$$
Since $V_F=\mathcal U^{s-1}$, choose $u\in\F_q$ with $F(u)\ne0$. Both $F(u)$ and $((LH^Q)(u))^{s-1}$ are nonzero $(s-1)$-st powers in $\F_q$, so $\alpha=a^{s-1}$ for some $a\in\F_q^\times$. Thus $F=G^{s-1}$, where $G=aLH^Q$. By Lemma~\ref{lem:br-v}, we obtain $G\in\mathcal P(\mathcal U,q)$, completing this case.

\subsection{Proof of Theorem~\ref{thm:explicit}}
First, suppose that $F(x)=(A(P(x)))^v$, with $A,P,v$
as in Theorem~\ref{thm:explicit}. Set $d_P=\deg P$ and
$d_A=\deg A$, so that
$\deg F=vd_Ad_P$.

Consider the $\F_{p^k}$-linear map $\varphi:\F_{q'}\to\F_q$ given by $\varphi(z)=A(z)$. Since $A(x)$ divides $x^{q'}-x$, we have $\#\ker\varphi=d_A$. The rank-nullity theorem then gives $\#A(\F_{q'})=q'/d_A$.
Since $A(\F_{q'})$ is an $\F_{p^k}$-vector space and $v\mid p^k-1$, each nonzero element of $A(\F_{q'})^v$ has exactly $v$ preimages in $A(\F_{q'})$ under the map $z\mapsto z^v$. Consequently, $\#V_F=1+(q'/d_A-1)/v$. Since $1\le v<q'/(2d_A)$, we obtain $\#V_F>2$.

Set $t=(q'/d_A-1)/v$, which is an integer because $q'/d_A$ is a power of $p^k$ and $v\mid p^k-1$. Since $P$ is an MVSP with $\#V_P=q'$, we have $q'-1<q/d_P\le q'$. Therefore,
\[
t\le t+\frac{d_A-1}{vd_A}
=\frac{q'-1}{vd_A}
<\frac{q}{\deg F}
\le\frac{q'}{vd_A}
=t+\frac1v
\le t+1.
\]
Thus $\lceil q/\deg F\rceil=t+1=\#V_F$, proving that $F$ is an MVSP.

Conversely, let $F\in\F_q[x]$ be an MVSP with $\#V_F>2$. By Lemma~\ref{lem:br-1}, up to affine changes of variables and values, we may assume that $V_F=\mathcal U^\omega$ for some $\F_{p^k}$-vector subspace $\mathcal U\subseteq\F_q$ containing $1$ and some positive divisor $\omega$ of $p^k-1$, where $\F_{p^k}\subseteq\F_q$.

If $\mathcal U$ is a subfield of $\F_q$, Theorem~\ref{thm:conjecture} gives $F=H^t$, where $H$ is an MVSP with value set $\F_{p^e}\subseteq\F_q$ and $t\mid p^e-1$. Since $\#V_F=(p^e-1)/t+1>2$, we have $t<p^e/2$. We take $P=H$, $\F_{q'}=\F_{p^e}$, $A(x)=x$, and $v=t<q'/(2\deg A)$.

If $\mathcal U$ is not a subfield of $\F_q$, Theorem~\ref{thm:conjecture} gives $F=H^\omega$, where $H\in\F_q[x]$ is an MVSP with $V_H=\mathcal U$. By Theorem~\ref{thm:subspace} and Remark~\ref{rem:ad-div}, there exist $A,P\in\F_q[x]$ such that
$$
H(x)=A(P(x)),
$$
where $P$ is an MVSP with value set $\F_{q'}\subseteq\F_q$, and $A$ is an $\F_{p^k}$-additive polynomial dividing $x^{q'}-x$, with $\F_{p^k}\subseteq\F_{q'}$.
Thus
$$
F(x)=A(P(x))^\omega.
$$
We take $v=\omega$. Since $\#V_F=1+(q'/\deg A-1)/\omega>2$, we have $2\omega\le q'/\deg A-1$, hence $v<q'/(2\deg A)$. This completes the proof.






\section{An Application to Frobenius nonclassical curves}
\label{sec:fnc}

Let $q=p^n$ be a prime power. A geometrically irreducible plane curve $\mathcal C$ defined over $\F_q$ is called $\F_q$-Frobenius nonclassical if the $q$-power Frobenius image of every nonsingular point $P\in\mathcal C(\overline{\F}_q)$ lies on the tangent line at $P$. These curves play a
central role in the St\"ohr--Voloch theory \cite[Chapter~8]{HKT},\cite{SV}
and include important families with many rational
points \cite{BoHo,HV}.

For curves with separated variables, Borges \cite{Bor0} established a connection between Frobenius nonclassicality and minimal value set polynomials. More precisely, let $F,G\in\F_q[T]$ be nonconstant, assume that $F(X)-G(Y)\notin\F_q[X^p,Y^p]$, and suppose that all geometrically irreducible components of the curve $F(X)-G(Y)=0$ are defined over $\F_q$.
If these components are $\F_q$-Frobenius nonclassical, then $F$ and $G$ are MVSPs with the same value set. The converse holds whenever their common value set has more than two elements. Consequently, Theorem~\ref{thm:explicit}, together with the small-value-set cases treated in \cite{Bor0}, provides a description of all such pairs.

We illustrate this consequence for curves of the form $Y^m=f(X)$, whose absolute irreducibility can be checked directly from the root multiplicities of $f$. In this case, the description reduces to the sets $\mathcal P(\F_s,q)$, with $\F_s\subseteq\F_q$, explicitly classified in Theorem~\ref{thm:b-c}.

\begin{corollary}\label{cor:fnc-kummer}
Let $\mathcal C:Y^m=f(X)$ be an absolutely irreducible
plane curve over $\F_q$, where $q=p^n$ and $1<m<q-1$,
and set $\F_s=\F_p(f(\F_q))$.
Then $\mathcal C$ is $\F_q$-Frobenius nonclassical
if and only if
\[
m=\frac{q-1}{s-1}
\qquad\text{and}\qquad
f\in\mathcal P(\F_s,q).
\]
\end{corollary}

\begin{proof}
By Theorem~\ref{thm:explicit}, the characterization in \cite[Theorem~2.8]{BR} holds unconditionally: $\mathcal C$ is $\F_q$-Frobenius nonclassical if and only if there exists a subfield $k\subseteq\F_q$ such that $m=(q-1)/(\#k-1)$ and $f\in\mathcal P(k,q)$.

Suppose first that $\mathcal C$ is $\F_q$-Frobenius nonclassical. Since $f(\F_q)=k$, we have $k=\F_p(f(\F_q))=\F_s$. Hence $m=(q-1)/(s-1)$ and $f\in\mathcal P(\F_s,q)$. Conversely, if $m=(q-1)/(s-1)$ and $f\in\mathcal P(\F_s,q)$, then \cite[Theorem~2.8]{BR} applies with $k=\F_s$, proving that $\mathcal C$ is $\F_q$-Frobenius nonclassical.

\end{proof}


\begin{thebibliography}{99}
\bibitem{BBQ21} D.~Bartoli, H.~Borges, and L.~Quoos,
\newblock Rational functions with small value set,
\newblock J. Algebra \textbf{565} (2021), 675--690.

\bibitem{BSD} B.~J.~Birch and H.~P.~F.~Swinnerton-Dyer,
\newblock Note on a problem of Chowla,
\newblock Acta Arith. \textbf{5} (1959), 417--423.

\bibitem{Bor0} H.~Borges,
\newblock Frobenius nonclassical components of curves with separated variables,
\newblock J. Number Theory \textbf{159} (2016), 402--426.

\bibitem{Bor2} H.~Borges and R.~Concei\c{c}\~ao,
\newblock On the characterization of minimal value set polynomials,
\newblock J. Number Theory \textbf{133} (2013), 2021--2035.

\bibitem{Bor1} H.~Borges and R.~Concei\c{c}\~ao,
\newblock A new family of Castle and Frobenius nonclassical curves,
\newblock J. Pure Appl. Algebra \textbf{222} (2018), 994--1002.

\bibitem{BoHo} H.~Borges and M.~Homma,
\newblock Points on singular Frobenius nonclassical curves,
\newblock Bull. Braz. Math. Soc. (N.S.) \textbf{48} (2017), 93--101.

\bibitem{BR21} H.~Borges and L.~Reis,
\newblock Minimal value set polynomials over fields of size $p^3$,
\newblock Proc. Amer. Math. Soc. \textbf{149} (2021), no.~9, 3639--3649.

\bibitem{BR} H.~Borges and L.~Reis,
\newblock Minimal value set polynomials,
\newblock Adv. Math. \textbf{482} (2025), Part B, Article 110627.

\bibitem{Bor3} H.~Borges, A.~Sep\'ulveda, and G.~Tizzotti,
\newblock Weierstrass semigroup and automorphism group of the $X_{n,r}$ curves,
\newblock Finite Fields Appl. \textbf{36} (2015), 121--132.

\bibitem{Carlitz1955} L.~Carlitz,
\newblock On the number of distinct values of a polynomial with coefficients in a finite field,
\newblock Proc. Japan Acad. \textbf{31} (1955), 119--120.

\bibitem{CLMS} L.~Carlitz, D.~J.~Lewis, W.~H.~Mills, and E.~G.~Straus,
\newblock Polynomials over finite fields with minimal value sets,
\newblock Mathematika \textbf{8} (1961), 121--130.

\bibitem{Cohen} S.~D.~Cohen,
\newblock The distribution of polynomials over finite fields,
\newblock Acta Arith. \textbf{17} (1970), 255--271.

\bibitem{DXZ} X.~Deng, B.~Xu, and Q.~Zhu,
\newblock The small value set polynomials over finite fields and monodromy groups,
\newblock J. Number Theory \textbf{276} (2025), 139--161.

\bibitem{Elkies} N.~D.~Elkies,
\newblock Curves with many points,
\newblock preprint.

\bibitem{Garcia} A.~Garcia,
\newblock The curves $y^n=f(x)$ over finite fields,
\newblock Arch. Math. (Basel) \textbf{54} (1990), no.~1, 36--44.

\bibitem{Cald1} J.~Gomez-Calderon,
\newblock A note on polynomials with minimal value set over finite fields,
\newblock Mathematika \textbf{35} (1988), 144--148.

\bibitem{Cald2} J.~Gomez-Calderon and D.~J.~Madden,
\newblock Polynomials with small value sets over finite fields,
\newblock J. Number Theory \textbf{28} (1988), 167--188.

\bibitem{GPTU} M.~Giulietti, F.~Pambianco, F.~Torres, and E.~Ughi,
\newblock On complete arcs arising from plane curves,
\newblock Des. Codes Cryptogr. \textbf{25} (2002), 237--246.

\bibitem{GRZ} R.~M.~Guralnick, J.~E.~Rosenberg, and M.~E.~Zieve,
\newblock A new family of exceptional polynomials in characteristic two,
\newblock Ann. of Math. (2) \textbf{172} (2010), 1361--1390.

\bibitem{GW} R.~M.~Guralnick and D.~Wan,
\newblock Bounds for fixed point free elements in a transitive group and applications to curves over finite fields,
\newblock Israel J. Math. \textbf{101} (1997), 255--287.

\bibitem{GZ} R.~M.~Guralnick and M.~E.~Zieve,
\newblock Polynomials with $\mathrm{PSL}(2)$ monodromy,
\newblock Ann. of Math. (2) \textbf{172} (2010), 1315--1359.


\bibitem{HV} A.~Hefez and J.~F.~Voloch,
\newblock Frobenius non-classical curves,
\newblock Arch. Math. (Basel) \textbf{54} (1990), 263--273.


\bibitem{HKT} J.~W.~P.~Hirschfeld, G.~Korchm\'aros, and F.~Torres,
\newblock Algebraic Curves over a Finite Field,
\newblock Princeton University Press, Princeton and Oxford, 2008.

\bibitem{Nie} R.~Lidl and H.~Niederreiter,
\newblock Finite Fields,
\newblock 2nd ed., Encyclopedia of Mathematics and its Applications,
Cambridge University Press, Cambridge, 1997.

\bibitem{Mills} W.~H.~Mills,
\newblock Polynomials with minimal value sets,
\newblock Pacific J. Math. \textbf{14} (1964), 225--241.

\bibitem{Mullen} G.~L.~Mullen and D.~Panario,
\newblock Handbook of Finite Fields,
\newblock CRC Press, Boca Raton, 2013.

\bibitem{MWW} G.~L.~Mullen, D.~Wan, and Q.~Wang,
\newblock Value sets of polynomial maps over finite fields,
\newblock Q. J. Math. \textbf{64} (2013), 1191--1196.

\bibitem{MRW} A.~Muratovi\'c-Ribi\'c and Q.~Wang,
\newblock On a conjecture of polynomials with prescribed range,
\newblock Finite Fields Appl. \textbf{18} (2012), 728--737.

\bibitem{MunueraSepulvedaTorresCodes}
C.~Munuera, A.~Sep\'ulveda, and F.~Torres,
\newblock Algebraic geometry codes from Castle curves,
\newblock in \emph{Coding Theory and Applications},
Lecture Notes in Comput. Sci. \textbf{5228},
Springer, 2008, 117--127.

\bibitem{MunueraSepulvedaTorresCastle}
C.~Munuera, A.~Sep\'ulveda, and F.~Torres,
\newblock Castle curves and codes,
\newblock Adv. Math. Commun. \textbf{3} (2009), no.~4, 399--408.


\bibitem{SV} K.-O.~St\"ohr and J.~F.~Voloch,
\newblock Weierstrass points and curves over finite fields,
\newblock Proc. London Math. Soc. (3) \textbf{52} (1986), 1--19.

\bibitem{Toledano} R.~Toledano,
\newblock A link between minimal value set polynomials and tamely ramified towers of function fields over finite fields,
\newblock J. Algebra Appl. \textbf{21} (2022), 2250189.

\bibitem{Wan} D.~Wan,
\newblock A p-adic lifting lemma and its applications to permutation polynomials,
\newblock Finite Fields Appl. \textbf{3} (1997), 46--69.

\bibitem{WSC} D.~Wan, P.~Shiue, and C.~S.~Chen,
\newblock Value sets of polynomials over finite fields,
\newblock Proc. Amer. Math. Soc. \textbf{119} (1993), 711--717.
\end{thebibliography}
 \end{document}